\documentclass[twocolumn,10pt,oneside]{amsart}
\usepackage[T1]{fontenc}
\usepackage[utf8]{inputenc}
\usepackage{textcomp}
\usepackage[noBBpl]{mathpazo}
\usepackage{eulervm}
\usepackage{graphicx}
\usepackage{overpic}
\usepackage{xcolor}
\definecolor{col1}{rgb}{0.75,0,0}
\definecolor{col2}{rgb}{0,0,0.75}
\usepackage{amsmath,amssymb,amsthm}
\usepackage{microtype}
\usepackage[numbers,sort]{natbib}
\usepackage{url}
\usepackage[left=25mm,right=25mm,top=20mm,bottom=30mm]{geometry}

\newcommand*{\ISet}{\mathcal{I}}
\renewcommand*{\vec}[1]{\mathbf{#1}}
\newcommand*{\defeq}{\mathrel{\mathop:}=}
\newcommand*{\tet}[1]{\mathbf{#1}}
\newcommand*{\tri}[1]{\tet{#1}}

\newtheorem{theorem}{Theorem}
\newtheorem{proposition}[theorem]{Proposition}
\newtheorem{corollary}[theorem]{Corollary}
\theoremstyle{definition}
\newtheorem{definition}[theorem]{Definition}

\title{Orthologic Tetrahedra with Intersecting Edges}
\author{Hans-Peter Schröcker}
\address{Hans-Peter Schröcker, Unit Geometry and CAD, University
  Innsbruck, Technikerstraße 13, A6020 Innsbruck, Austria}
\urladdr{http://geometrie.uibk.ac.at/schroecker/}

\begin{document}
\twocolumn[%
\begin{abstract}
Two tetrahedra are called orthologic if the lines through vertices of
one and perpendicular to corresponding faces of the other are
intersecting. This is equivalent to the orthogonality of
non-corresponding edges. We prove that the additional assumption of
intersecting non-corresponding edges (``orthosecting tetrahedra'')
implies that the six intersection points lie on a sphere. To a given
tetrahedron there exists generally a one-parametric family of
orthosecting tetrahedra. The orthographic projection of the locus of
one vertex onto the corresponding face plane of the given tetrahedron
is a curve which remains fixed under isogonal conjugation. This allows
the construction of pairs of conjugate orthosecting tetrahedra to a
given tetrahedron.



%
\end{abstract}
\maketitle]

\section{Introduction}
\label{sec:introduction}

Ever since the introduction of orthologic triangles and tetrahedra by
J.~Steiner in 1827 \cite{steiner27:_lehrsaetze} these curious pairs
have attracted researchers in elementary geometry. The characterizing
property of orthologic tetrahedra is concurrency of the straight lines
through vertices of one tetrahedron and perpendicular to corresponding
faces of the second. Alternatively, one can say that non-corresponding
edges are orthogonal. Proofs of fundamental properties can be found in
\cite{neuberg84:_memoire_tetraedre} and
\cite{neuberg07:_orthologische_tetraeder}. Quite a few results are
known on special families of orthologic triangles and tetrahedra. See
for example \cite{servais23:_trois_tetraedres,%
  thebault52:_perspective_orthologic_triangles,%
  mandan79:_orthologic_desargues_figure,%
  mandan79:_semi_bilogic_bilogic_tetrahedra} for more information on
orthologic tetrahedra (or triangles) which are also perspective or
\cite{goormaghtigh29:_orthopole_theorem} for a generalization of a
statement on families of orthologic triangles related to orthopoles.

In this article we are concerned with \emph{orthosecting
  tetrahedra}\,---\,orthologic tetrahedra such that
non-cor\-re\-spon\-ding edges intersect orthogonally. The concept as
well as a few basic results will be introduced in
Section~\ref{sec:preliminaries}. In
Section~\ref{sec:intersection-points} we show that the six
intersection points of non-corresponding edges necessarily lie on a
sphere (or a plane). While the computation of orthosecting pairs
requires, in general, the solution of a system of algebraic equations,
conjugate orthosecting tetrahedra can be constructed from a given
orthosecting pair. This is the topic of
Section~\ref{sec:solution-family}. Our treatment of the subject is of
elementary nature. The main ingredients in the proofs come from
descriptive geometry and triangle geometry.

A few words on notation: By $A_1A_2A_3$ we denote the triangle with
vertices $A_1$, $A_2$, and $A_3$, by $A_1A_2A_3A_4$ the tetrahedron
with vertices $A_1$, $A_2$, $A_3$, and $A_4$. The line spanned by two
points $A_1$ and $A_2$ is $A_1 \vee A_2$, the plane spanned by three
points $A_1$, $A_2$, and $A_3$ is $A_1 \vee A_2 \vee
A_3$. Furthermore, $\ISet_n$ denotes the set of all $n$-tuples with
pairwise different entries taken from the set $\{1,\ldots,n\}$.

\section{Preliminaries}
\label{sec:preliminaries}

Two triangles $A_1A_2A_3$ and $B_1B_2B_3$ are called
\emph{orthologic}, if the three lines
\begin{equation}
  \label{eq:1}
  a_i\colon A_i \in a_i,\
  a_i \perp B_j \vee B_k;\quad
  (i,j,k) \in \ISet_3
\end{equation}
intersect in a point $O_A$, the \emph{orthology center} of $A_1A_2A_3$
with respect to $B_1B_2B_3$. In this case, also the lines
\begin{equation}
  \label{eq:2}
  b_i\colon B_i \in b_i,\
  b_i \perp A_j \vee A_k;\quad
  (i,j,k) \in \ISet_3
\end{equation}
intersect in a point $O_B$, the orthology center of $B_1B_2B_3$ with
respect to $A_1A_2A_3$. The concept of orthologic tetrahedra is
similar. Two tetrahedra $\tet{A} = A_1A_2A_3A_4$ and $\tet{B} =
B_1B_2B_3B_4$ are called \emph{orthologic}, if the four lines
\begin{equation}
  \label{eq:3}
  a_i\colon A_i \in a_i,\
  a_i \perp B_j \vee B_k \vee B_l;\quad
  (i,j,k,l) \in \ISet_4
\end{equation}
intersect in a point $O_A$, the orthology center of $\tet{A}$ with
respect to $\tet{B}$. In this case, also the lines
\begin{equation}
  \label{eq:4}
  b_i\colon B_i \in b_i,\
  b_i \perp A_j \vee A_k \vee A_l;\quad
  (i,j,k,l) \in \ISet_4
\end{equation}
intersect in a $O_B$, the orthology center of $\tet{B}$ with respect
to $\tet{A}$. Orthologic triangles and tetrahedra have been introduced
by J.~Steiner in \cite{steiner27:_lehrsaetze}; proofs of fundamental
properties can be found in
\cite{neuberg84:_memoire_tetraedre,neuberg07:_orthologische_tetraeder}
or
\cite[pp.~173--174]{altshiller-court64:_modern_pure_solid_geometry}.

The symmetry of the two tetrahedra in the definition of orthology is a
consequence of the following alternative characterization of
orthologic tetrahedra. It is well-known but we give a proof which
introduces concepts and techniques that will frequently be employed
throughout this paper.

\begin{proposition}
  \label{prop:1}
  The two tetrahedra $\tet{A}$ and $\tet{B}$ are orthologic if and
  only if non-corresponding edges are orthogonal:
  \begin{equation}
    \label{eq:5}
    A_i \vee A_j \perp B_k \vee B_l,\quad
    (i,j,k,l) \in \ISet_4.
  \end{equation}
\end{proposition}

\begin{proof}
  We only require that the lines $a_i$ through $A_i$ and orthogonal to
  the plane $B_j \vee B_k \vee B_l$ intersect in a point $O_A$. The
  plane $A_i \vee A_j \vee O_A$ contains two lines orthogonal to the
  line $B_k \vee B_l$, $(i,j,k,l) \in \ISet_4$. Therefore, all lines
  in this plane, in particular $A_i \vee A_j$, are orthogonal to $B_k
  \vee B_l$.

  Assume conversely that the orthogonality conditions \eqref{eq:5}
  hold. Clearly, any two perpendiculars $a_i$ intersect. We have to
  show that all intersection points $A_{ij} = a_i \cap a_j$
  coincide. Using our freedom to translate the tetrahedron $\tet{A}$
  without destroying orthogonality relations we can ensure, without
  loss of generality, the existence of the intersection points
  \begin{equation}
    \label{eq:6}
    \begin{gathered}
      V_{12} \defeq (A_1 \vee A_2) \cap (B_3 \vee B_4),\\
      V_{13} \defeq (A_1 \vee A_3) \cap (B_2 \vee B_4),\\
      V_{23} \defeq (A_2 \vee A_3) \cap (B_1 \vee B_4).
    \end{gathered}
  \end{equation}
  Consider now the orthographic projection onto the plane $A_1 \vee
  A_2 \vee A_3$ (Figure~\ref{fig:orthographic-projection}). We denote
  the projection of a point $X$ by $X'$. By the Right-Angle Theorem of
  descriptive geometry,\footnote{In an orthographic projection the
    right angle between two lines appears as right angle if and only
    if one line is in true size (parallel to the image plane) and the
    other is not in a point-view (orthogonal to the image plane).}
  the points $B_1'$, $B_2'$ and $B_3'$ lie on the perpendiculars
  through $B_4'$ onto the sides of the triangle $\tri{A} =
  A_1A_2A_3$. Moreover, since the plane $V_{12} \vee V_{13} \vee
  V_{23}$ appears in true size, the lines $a_i$ through $A_i$ and
  orthogonal to the respective face planes of $\tet{A}$ have
  projections $a'_1$, $a'_2$, $a'_3$ orthogonal to the edges of the
  triangle $\tri{V} = V_{23}V_{13}V_{12}$. The triangles $\tri{V}$ and
  $\tri{A}$ are orthologic. Therefore the lines $a'_i$ intersect in a
  point $O'_A$ which is necessarily the projection of a common point
  $O_A$ of the lines $a_1$, $a_2$, and~$a_3$.
\end{proof}

\begin{figure}
  \centering
  \includegraphics{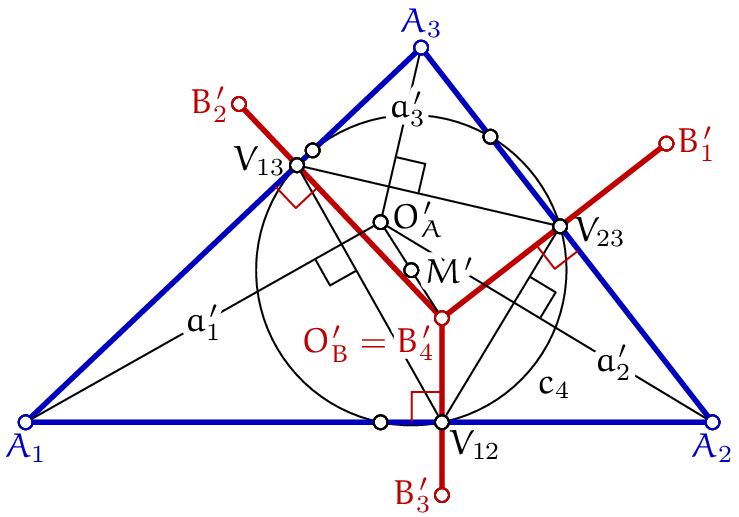}
  \caption{Orthographic projection onto a face plane}
  \label{fig:orthographic-projection}
\end{figure}

\section{The six intersection points}
\label{sec:intersection-points}

The new results in this paper concern pairs of orthologic tetrahedra
$\tet{A} = A_1A_2A_3A_4$ and $\tet{B} = B_1B_2B_3B_4$ such that
non-corresponding edges are not only orthogonal but also
intersecting. That is, in addition to \eqref{eq:5} we also require
existence of the points
\begin{equation}
  \label{eq:7}
  V_{ij} \defeq (A_i \vee A_j) \cap (B_k \vee B_l) \neq \varnothing,\quad
  (i,j,k,l) \in \ISet_4.
\end{equation}

\begin{definition}
  \label{def:2}
  We call two tetrahedra $\tet{A}$ and $\tet{B}$ \emph{orthosecting}
  if their vertices can be labelled as $A_1A_2A_3A_4$ and
  $B_1B_2B_3B_4$, respectively, such that \eqref{eq:5} and
  \eqref{eq:7} hold.
\end{definition}

\begin{theorem}
  \label{th:3}
  If two tetrahedra are orthosecting, the six intersection points of
  non-corresponding edges lie on a sphere (or a plane, if flat
  tetrahedra are permitted). The sphere center is the midpoint between
  the orthology centers.
\end{theorem}

\begin{proof}
  Denote the two tetrahedra by $\tet{A} = A_1A_2A_3A_4$ and $\tet{B} =
  B_1B_2B_3B_4$ such that the lines $A_i \vee A_j$ and $B_k \vee B_l$
  intersect orthogonally in $V_{ij}$ for $(i,j,k,l) \in \ISet_4$. As
  in the proof of Proposition~\ref{prop:1} we consider the
  orthographic projection onto the plane $A_1 \vee A_2 \vee A_3$
  (Figure~\ref{fig:orthographic-projection}). Clearly, $B_4'$ equals
  the projection $O_B'$ of the orthology center $O_B$ of $\tet{B}$
  with respect to $\tet{A}$. If it lies on the circumcircle of
  $A_1A_2A_3$, all perpendiculars from $B_4'$ onto the sides of
  $A_1A_2A_3$ are parallel. In this case the tetrahedron
  $B_1B_2B_3B_4$ is flat and the theorem's statement holds. Otherwise,
  the points $V_{12}$, $V_{13}$, and $V_{23}$ define a circle
  $c_4$\,---\,the pedal circle of the point $B'_4$ with respect to the
  triangle $A_1A_2A_3$.  By the Right-Angle Theorem the projection
  $O'_A$ of the orthology center $O_A$ of $\tet{A}$ with respect to
  $\tet{B}$ is the orthology center of the triangle $A_1A_2A_3$ with
  respect to the triangle $V_{23}V_{13}V_{12}$. Moreover, from
  elementary triangle geometry it is known that the center $M'$ of
  $c_4$ halves the segment between $B'_4$ and $O'_A$
  \cite[pp.~54--56]{honsberger91:_mathemical_morsels}. Hence all
  circles $c_i$ drawn in like manner on the faces of $\tet{A}$ have
  axes which intersect in the midpoint $M$ of the two orthology
  centers $O_A$ and $O_B$. Moreover, any two of these circles share
  one of the points $V_{ij}$. Hence, these circles are co-spherical
  and the proof is finished.
\end{proof}

The proof of Theorem~\ref{th:3} can also be applied to a slightly more
general configuration where only five of the six edges intersect
orthogonally. We formulate this statement as a corollary:

\begin{corollary}
  \label{cor:4}
  If $\tet{A} = A_1A_2A_3A_4$ and $\tet{B} = B_1B_2B_3B_4$ are two
  orthologic tetrahedra such that five non-corresponding edges
  intersect, the five intersection points lie on a sphere (or a
  plane).
\end{corollary}

\section{The one-parametric family of solution tetrahedra}
\label{sec:solution-family}

So far we have dealt with properties of a pair of orthosecting
tetrahedra but we have left aside questions of existence or
computation. In this section $\tet{A} = A_1A_2A_3A_4$ is a given
tetrahedron to which an orthosecting tetrahedron $\tet{B} =
B_1B_2B_3B_4$ is sought.

\subsection{Construction of orthologic tetrahedra}
\label{sec:construction-orthologic}

At first, we consider the simpler case of orthologic pairs. Clearly,
translation of the face planes of $\tet{B}$ will transform an
orthologic tetrahedron into a like tetrahedron (unless all planes pass
through a single point). Therefore, we consider tetrahedra with
parallel faces as equivalent.

The maybe simplest construction of an equivalence class of solutions
consists of the choice of the orthology center $O_A$. This immediately
yields the face normals $\vec{n}_i$ of $\tet{B}$ as connecting vectors
of $O_A$ and $A_i$. The variety of solution classes is of dimension
three, one solution to every choice of $O_A$. Since five edges
determine two face planes of a tetrahedron and, in case of suitable
orthogonality relations, also the orthology center $O_A$, we obtain

\begin{theorem}
  \label{th:5}
  If the vertices of two tetrahedra can be labelled such that five
  non-corresponding pairs of edges are orthogonal then so is the
  sixth.
\end{theorem}

The variety of all solution classes contains a two-parametric set of
trivial solutions $\vec{n}_1 = \vec{n}_2 = \vec{n}_3 =
\vec{n}_4$. They correspond to orthology centers at infinity, the
solution tetrahedra are flat. Note that the possibility to label the
edges such that non-corresponding pairs are orthogonal is essential
for the existence of non-flat solutions. If, for example,
corresponding edges are required to be orthogonal only flat solutions
exist.

\subsection{Conjugate pairs of orthosecting tetrahedra}
\label{sec:conjugate-pairs}

Establishing algebraic equations for solution tetrahedra is
straightforward. Six orthogonality conditions and six intersection
condition result in a system of six linear and six quadratic equations
in the twelve unknown coordinates of the vertices of
$\tet{B}$. Because of Theorem~\ref{th:5}, only five of the six linear
orthogonality conditions are independent. Therefore, we can expect a
one-dimensional variety of solution tetrahedra. This expectation is
generically true, as can be confirmed by computing the dimension of
the ideal spanned by the orthosecting conditions by means of a
computer algebra system.

The numeric solution of the system induced by the orthosecting
conditions poses no problems. We used the software
Bertini,\footnote{D.~J.~Bates, J.~D.~Hauenstein, A.~J.~Sommese,
  Ch.~W.~Wampler: Bertini: Software for Numerical Algebraic Geometry,
  \url{http://www.nd.edu/~sommese/bertini/}.} for that purpose. Symbolic
approaches are feasible as well. One of them will be described in
Subsection~\ref{sec:computational-issues}.  It is based on a curious
conjugacy which can be defined in the set of all tetrahedra that
orthosect the given tetrahedron~$\tet{A}$.

Assume that $\tet{B} = B_1B_2B_3B_4$ is a solution tetrahedron and
denote the orthographic projection of $B_i$ onto the face plane $A_j
\vee A_k \vee A_l$ by $B^\star_i$, $(i,j,k,l) \in \ISet_4$. By the
Right-Angle theorem the pedal points of all points $B^\star_i$ on the
edges of $A_jA_kA_l$ are precisely the intersection points defined in
\eqref{eq:7}. Three intersection points on the same face of $\tet{A}$
form a pedal triangle. This observation gives rise to

\begin{definition}
  \label{def:6}
  A \emph{pedal chain} on a tetrahedron is a set of four pedal
  triangles, each with respect to one face triangle of the
  tetrahedron, such that any two pedal triangles share the vertex on
  the common edge of their faces (Figure~\ref{fig:pedal-chain}). If
  all vertices of pedal triangles lie on a sphere (or a plane), we
  speak of a \emph{spherical pedal chain}.
\end{definition}

\begin{figure}
  \centering
  \includegraphics{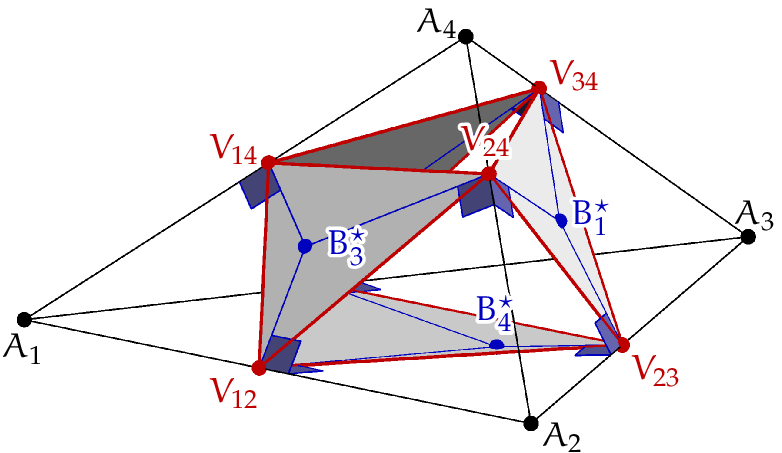}
  \caption{A pedal chain.}
  \label{fig:pedal-chain}
\end{figure}

If $A_1A_2A_3A_4$ and $B_1B_2B_3B_4$ are orthosecting, the proof of
Theorem~\ref{th:3} shows that six intersection points are the vertices
of a spherical pedal chain. The converse is also true:

\begin{theorem}
  \label{th:7}
  Given the vertices $V_{ij}$ of a spherical pedal chain on a
  tetrahedron $\tet{A} = A_1A_2A_3A_4$ there exists a unique
  orthosecting tetrahedron $\tet{B} = B_1B_2B_3B_4$ such that $A_i
  \vee A_j \cap B_k \vee B_l = V_{ij}$ for all $(i,j,k,l) \in
  \ISet_4$.
\end{theorem}

\begin{proof}
  If a solution tetrahedron $\tet{B}$ exists at all it must be unique
  since its faces lie in the planes $\beta_i \defeq V_{ij} \vee V_{ik}
  \vee V_{il}$ ($i$, $j$, $k \in \{1,2,3,4\}$ pairwise
  different).\footnote{The case of collinear or coinciding points
    $V_{ij}$ leads to degenerate solution tetrahedra whose faces
    contain one vertex of~$\tet{A}$.}

  \begin{figure}
    \centering
    \includegraphics{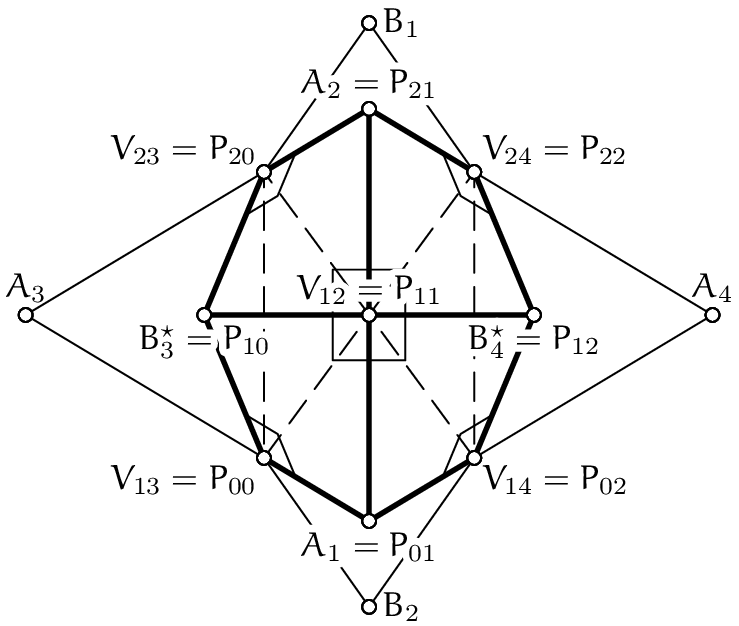}
    \caption{Proof of Theorem~\ref{th:7}}
    \label{fig:5-orthosecting}
  \end{figure}

  In order to prove existence, we have to show that the lines $A_i
  \vee A_j$ and $\beta_i \cap \beta_j$ are, indeed, orthogonal for all
  pairwise different $i$, $j \in \{1,2,3,4\}$. We denote the point
  from which the pedal triangle on the face $A_iA_jA_k$ originates
  (the ``anti-pedal point'') by $B^\star_l$ and show orthogonality
  between $A_1 \vee A_2$ and $\beta_1 \cap \beta_2$ for $(i,j,k,l) \in
  \ISet_4$. Relabelling according to
  \begin{equation}
    \label{eq:8}
    \begin{aligned}
      &P_{00} \defeq V_{13},&
      &P_{01} \defeq A_1,&
      &P_{02} \defeq V_{14},\\
      &P_{10} \defeq B^\star_4,&
      &P_{11} \defeq V_{12},&
      &P_{12} \defeq B^\star_3,\\
      &P_{20} \defeq V_{23},&
      &P_{21} \defeq A_2,&
      &P_{22} \defeq V_{24}
    \end{aligned}
  \end{equation}
  (Figure~\ref{fig:5-orthosecting}) we obtain a net of points
  $P_{ij}$. In every elementary quadrilateral the angle measure at two
  opposite vertices equals $\pi/2$. Thus, the net is \emph{circular.}
  Such structures are extensively studied in the context of discrete
  differential geometry
  \cite{bobenko08:_discrete_differential_geometry}. Our case is rather
  special since two pairs of quadrilaterals span the same plane. This
  does, however, not hinder application of
  \cite[Theorem~4.21]{bobenko08:_discrete_differential_geometry} which
  states that our assumptions on the co-spherical (or co-planar)
  position of the points $P_{00}$, $P_{02}$, $P_{11}$, $P_{20}$, and
  $P_{22}$ is equivalent to the fact that the net $P_{ij}$ is a
  \emph{discrete isothermic net}. These nets have many remarkable
  characterizing properties. One of them, stated in
  \cite[Theorem~2.27]{bobenko08:_discrete_differential_geometry}, says
  that the planes $P_{00} \vee P_{11} \vee P_{02}$, $P_{10} \vee
  P_{11} \vee P_{12}$, and $P_{20} \vee P_{11} \vee P_{22}$ have a
  line in common. In our original notation this means that the line
  $\beta_1 \cap \beta_2$ intersects the face normal of $A_1 \vee A_2
  \vee A_3$ through $B^\star_4$ and the face normal of $A_1 \vee A_2
  \vee A_4$ through $B^\star_3$. Therefore, it is orthogonal to $A_1
  \vee A_2$.
\end{proof}

As a consequence of Theorem~\ref{th:7} it can be shown that tetrahedra
which orthosect $\tet{A}$ come in conjugate pairs: Given $\tet{A}$ and
an orthosecting tetrahedron $\tet{B}$ it is possible to construct a
second orthosecting tetrahedron $\tet{C}$. The same construction with
$\tet{C}$ as input yields the tetrahedron $\tet{B}$.  This conjugacy
is related to the pedal chain originating from $\tet{B}$. The key
ingredient is the following result from elementary triangle geometry
\cite[pp.~54--56]{honsberger91:_mathemical_morsels}:

\begin{figure}
  \centering
  \includegraphics{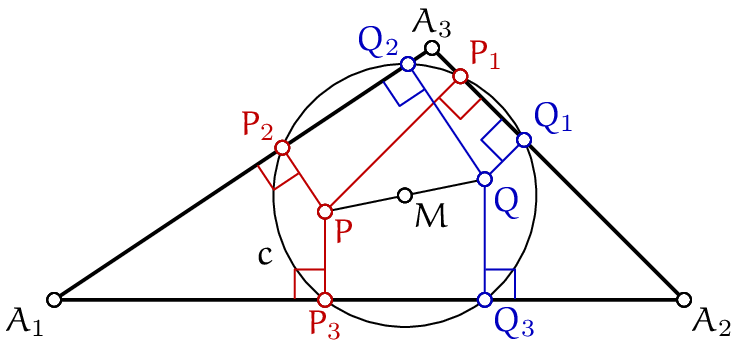}
  \caption{Pedal circles in a triangle}
  \label{fig:pedal-circle}
\end{figure}

\begin{proposition}
  \label{prop:8}
  If $P$ is a point in the plane of the triangle $A_1A_2A_3$ and $c$
  its pedal circle, the reflection $Q$ of $P$ in the center $M$ of $c$
  has the same pedal circle $c$ (Figure~\ref{fig:pedal-circle}).
\end{proposition}

Suppose that $\tet{A}$ and $\tet{B}$ are orthosecting tetrahedra. The
orthographic projections $B^\star_i$ of the vertices of $\tet{B}$ onto
corresponding face planes of $\tet{A}$ are points whose pedal
triangles form a spherical pedal chain. By reflecting $B^\star_i$ in
the centers of the pedal circles on the faces of $\tet{A}$ we obtain
points $C^\star_i$ which, according to Proposition~\ref{prop:8}, give
rise to a second spherical pedal chain (with the same sphere of
vertices) and, by Theorem~\ref{th:7}, can be used to construct a
second orthosecting tetrahedron~$\tet{C}$
(Figure~\ref{fig:conjugate-solutions}).

\begin{figure*}
  \centering
  \begin{minipage}{0.5\linewidth}
    \centering
    \includegraphics{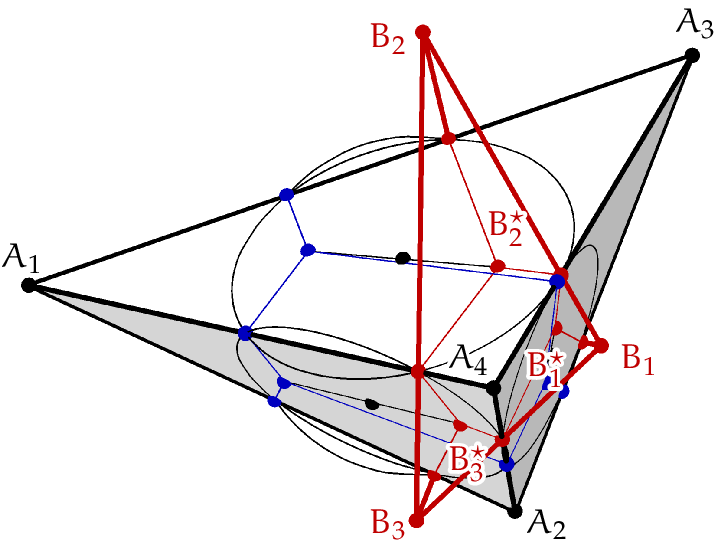}
  \end{minipage}%
  \begin{minipage}{0.5\linewidth}
    \centering
    \includegraphics{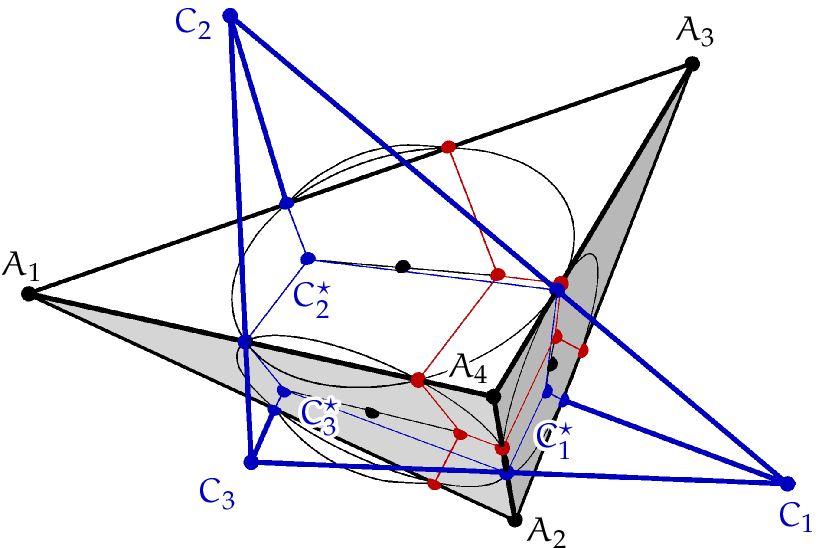}
  \end{minipage}
  \caption{A conjugate pair $\tet{B}$, $\tet{C}$ of orthosecting
    tetrahedra.}
  \label{fig:conjugate-solutions}
\end{figure*}

The points $P$ and $Q$ of Proposition~\ref{prop:8} are called
\emph{isogonal conjugates} with respect to the triangle $A_1A_2A_3$.
The above considerations lead immediately to

\begin{theorem}
  \label{th:9}
  Given a tetrahedron $\tet{A} = A_1A_2A_3A_4$, the orthographic
  projection of all vertices $B^\star_i$ of orthosecting tetrahedra
  onto the face plane $A_j \vee A_k \vee A_l$ of $\tet{A}$ (with
  $(i,j,k,l) \in \ISet_4$) is a curve which is isogonally
  self-conjugate with respect to the triangle $A_jA_kA_l$.
\end{theorem}

\subsection{Computational issues}
\label{sec:computational-issues}

We continue with a few remarks on the actual computation of the
isogonal self-conjugate curves of Theorem~\ref{th:9} with the help of
a computer algebra system. Our first result concerns the construction
of pedal chains.

\begin{theorem}
  \label{th:10}
  Consider a tetrahedron $\tet{A} = A_1A_2A_3A_4$ and six points
  $V_{ij} \in A_i \vee A_j$, $(i,j,k,l) \in \ISet_4$. If three of the
  four triangles $V_{ij}V_{jk}V_{ki}$, with $(i,j,k) \in \ISet_3$, are
  pedal triangles with respect to the triangle $A_iA_jA_k$ then this
  is also true for the fourth.
\end{theorem}

\begin{proof}
  Assume that the triangles $V_{12}V_{24}V_{14}$,
  $V_{23}V_{24}V_{34}$, and $V_{13}V_{34}V_{14}$ are pedal triangles
  of their respective face triangles. We have to show that
  $V_{12}V_{23}V_{13}$ is a pedal triangle of $A_1A_2A_3$.  As usual,
  the anti-pedal points are denoted by $B^\star_1$, $B^\star_2$, and
  $B^\star_3$.  Clearly, we have $B^\star_i \vee B^\star_j \perp A_k
  \vee A_4$ for $(i,j,k,4) \in \ISet_4$. Denote by $B^\circ_4$ a point
  in the intersection of the three planes incident with $V_{ij}$ and
  perpendicular to $A_i \vee A_j$, $(i,j,k) \in \ISet_3$. By
  Proposition~\ref{prop:1} the tetrahedra $\tet{A}$ and
  $B^\star_1B^\star_2B^\star_3B^\circ_4$ are orthologic. Therefore,
  the face normals $n_l$ of $A_i \vee A_j \vee A_k$ through
  $B^\star_l$ have a point $B_4$ in common ($l \neq 4$, $(i,j,k,l) \in
  \ISet_4$). By the Right-Angle Theorem, the intersection point
  $B^\star_4$ of the orthographic projections of $n_1$, $n_2$, and
  $n_3$ onto $A_1 \vee A_2 \vee A_3$ has $V_{12}V_{23}V_{13}$ as its
  pedal triangle.
\end{proof}

In order to construct a pedal chain on a tetrahedron $\tet{A} =
A_1A_2A_3A_4$ on can proceed as follows:
\begin{enumerate}
\item Prescribe an arbitrary pedal triangle, say $V_{12}V_{23}V_{13}$.
\item Choose one anti-pedal point, say $B^\star_3$, on a neighbouring
  face. It is restricted to the perpendicular to $A_1 \vee A_2$ trough
  $V_{12}$.
\item The remaining pedal points are determined. Theorem~\ref{th:10}
  guarantees that the final completion of $V_{34}$ is possible without
  contradiction.
\end{enumerate}

In order to construct a spherical pedal chain, the choice of
$B^\star_4$ and $B^\star_3$ needs to be appropriate. A simple
computation shows that there exist two possible choices (in algebraic
sense) for $B^\star_3$ such that the points $V_{12}$, $V_{13}$,
$V_{23}$, $V_{14}$, and $V_{24}$ are co-spherical (or
co-planar). Demanding that the remaining vertex $V_{34}$ lies on the
same sphere yields an algebraic condition on the coordinates of
$B^\star_4$\,---\,the algebraic equation of the isogonally
self-conjugate curve $i_4$ from Theorem~\ref{th:9}. We are currently
not able to carry out the last elimination step in full
generality. Examples suggest, however, that $i_4$ is of degree
nine. Once a point on $i_4$ is determined, the computation of the
corresponding orthosecting tetrahedron is trivial.

\section{Conclusion and future research}
\label{sec:conclusion}

We introduced the concept of orthosecting tetrahedra and presented a
few results related to them. In particular we characterized the six
intersection points as vertices of a spherical pedal chain on either
tetrahedron. This characterization allows the construction of
conjugate orthosecting tetrahedra to a given tetrahedron $\tet{A}$.

In general, there exists a one-parametric family of tetrahedra which
orthosect $\tet{A}$. The orthographic projection of their vertices on
the plane of a face triangle of $\tet{A}$ is an isogonally
self-conjugate algebraic curve. Maybe it is worth to study other loci
related to the one-parametric family of orthosecting tetrahedra. Since
every sphere that carries vertices of one pedal chain also carries the
vertices of a second pedal chain, the locus of their centers might
have a reasonable low algebraic degree.

Moreover, other curious properties of orthosecting tetrahedra seem
likely to be discovered. For example, the repeated construction of
conjugate orthosecting tetrahedra yields an infinite sequence $\langle
\tet{B}_n \rangle_{n \in \mathbb{Z}}$ of tetrahedra such that
$\tet{B}_{n-1}$ and $\tet{B}_{n+1}$ form a conjugate orthosecting pair
with respect to $\tet{B}_n$ for every $n \in \mathbb{Z}$. All
intersection points of non-corresponding edges lie on the same sphere
and only two points serve as orthology centers for any orthosecting
pair $\tet{B}_{n}$, $\tet{B}_{n+1}$. General properties and special
cases of this sequence might be a worthy field of further study.

Finally, we would like to mention two possible extensions of this
article's topic. It seems that, with exception of Steiner's result on
orthologic triangles on the sphere, little is known on orthologic
triangles and tetrahedra in non-Euclidean spaces. Moreover, one might
consider a relaxed ``orthology property'' as suggested by the
anonymous reviewer: It requires that the four lines $a_1$, $a_2$,
$a_3$, $a_4$ defined in \eqref{eq:3} lie in a regulus (and not
necessarily in a linear pencil). This concept is only useful if the
regulus position of the lines $a_i$ also implies regulus position of
the lines $b_j$ of \eqref{eq:4}. We have some numerical evidence that
this is, indeed, the case.



%



\end{document}